\documentclass[12pt]{article}
\usepackage{amssymb}
\usepackage{amsfonts}
\usepackage{german}
\long\def\ig#1{\relax}
\ig{Thanks to Roberto Minio for this def'n.  Compare the def'n of
\comment in AMSTeX.}
 
\newcount \coefa
\newcount \coefb
\newcount \coefc
\newcount\tempcounta
\newcount\tempcountb
\newcount\tempcountc
\newcount\tempcountd
\newcount\xext
\newcount\yext
\newcount\xoff
\newcount\yoff
\newcount\gap%
\newcount\arrowtypea
\newcount\arrowtypeb
\newcount\arrowtypec
\newcount\arrowtyped
\newcount\arrowtypee
\newcount\height
\newcount\width
\newcount\xpos
\newcount\ypos
\newcount\run
\newcount\rise
\newcount\arrowlength
\newcount\halflength
\newcount\arrowtype
\newdimen\tempdimen
\newdimen\xlen
\newdimen\ylen
\newsavebox{\tempboxa}%
\newsavebox{\tempboxb}%
\newsavebox{\tempboxc}%
 
\makeatletter
\def\settypes(#1,#2,#3){\arrowtypea#1 \arrowtypeb#2 \arrowtypec#3}
\def\settoheight#1#2{\setbox\@tempboxa\hbox{#2}#1\ht\@tempboxa\relax}%
\def\settodepth#1#2{\setbox\@tempboxa\hbox{#2}#1\dp\@tempboxa\relax}%
\def\settokens[#1`#2`#3`#4]{%
     \def\tokena{#1}\def\tokenb{#2}\def\tokenc{#3}\def\tokend{#4}}
\def\setsqparms[#1`#2`#3`#4;#5`#6]{%
\arrowtypea #1
\arrowtypeb #2
\arrowtypec #3
\arrowtyped #4
\width #5
\height #6
}
\def\setpos(#1,#2){\xpos=#1 \ypos#2}
 
\def\bfig{\begin{picture}(\xext,\yext)(\xoff,\yoff)}
\def\efig{\end{picture}}
 
\def\putbox(#1,#2)#3{\put(#1,#2){\makebox(0,0){$#3$}}}
 
\def\settriparms[#1`#2`#3;#4]{\settripairparms[#1`#2`#3`1`1;#4]}%
 
\def\settripairparms[#1`#2`#3`#4`#5;#6]{%
\arrowtypea #1
\arrowtypeb #2
\arrowtypec #3
\arrowtyped #4
\arrowtypee #5
\width #6
\height #6
}
 
\def\resetparms{\settripairparms[1`1`1`1`1;500]\width 500}%default values%
 
\resetparms
 
\def\mvector(#1,#2)#3{%%
\put(0,0){\vector(#1,#2){#3}}%
\put(0,0){\vector(#1,#2){30}}%
}
\def\evector(#1,#2)#3{{%%
\arrowlength #3
\put(0,0){\vector(#1,#2){\arrowlength}}%
\advance \arrowlength by-30
\put(0,0){\vector(#1,#2){\arrowlength}}%
}}
 
\def\horsize#1#2{%
\settowidth{\tempdimen}{$#2$}%
#1=\tempdimen
\divide #1 by\unitlength
}
 
\def\vertsize#1#2{%
\settoheight{\tempdimen}{$#2$}%
#1=\tempdimen
\settodepth{\tempdimen}{$#2$}%
\advance #1 by\tempdimen
\divide #1 by\unitlength
}
 
\def\vertadjust[#1`#2`#3]{%
\vertsize{\tempcounta}{#1}%
\vertsize{\tempcountb}{#2}%
\ifnum \tempcounta<\tempcountb \tempcounta=\tempcountb \fi
\divide\tempcounta by2
\vertsize{\tempcountb}{#3}%
\ifnum \tempcountb>0 \advance \tempcountb by20 \fi
\ifnum \tempcounta<\tempcountb \tempcounta=\tempcountb \fi
}
 
\def\horadjust[#1`#2`#3]{%
\horsize{\tempcounta}{#1}%
\horsize{\tempcountb}{#2}%
\ifnum \tempcounta<\tempcountb \tempcounta=\tempcountb \fi
\divide\tempcounta by20
\horsize{\tempcountb}{#3}%
\ifnum \tempcountb>0 \advance \tempcountb by60 \fi
\ifnum \tempcounta<\tempcountb \tempcounta=\tempcountb \fi
}
 
\ig{ In this procedure, #1 is the paramater that sticks out all the way,
#2 sticks out the least and #3 is a label sticking out half way.  #4 is
the amount of the offset.}
 
\def\sladjust[#1`#2`#3]#4{%
\tempcountc=#4
\horsize{\tempcounta}{#1}%
\divide \tempcounta by2
\horsize{\tempcountb}{#2}%
\divide \tempcountb by2
\advance \tempcountb by-\tempcountc
\ifnum \tempcounta<\tempcountb \tempcounta=\tempcountb\fi
\divide \tempcountc by2
\horsize{\tempcountb}{#3}%
\advance \tempcountb by-\tempcountc
\ifnum \tempcountb>0 \advance \tempcountb by80\fi
\ifnum \tempcounta<\tempcountb \tempcounta=\tempcountb\fi
\advance\tempcounta by20
}
 
\def\putvector(#1,#2)(#3,#4)#5#6{{%
\xpos=#1
\ypos=#2
\run=#3
\rise=#4
\arrowlength=#5
\arrowtype=#6
\ifnum \arrowtype<0
    \ifnum \run=0
        \advance \ypos by-\arrowlength
    \else
        \tempcounta \arrowlength
        \multiply \tempcounta by\rise
        \divide \tempcounta by\run
        \ifnum\run>0
            \advance \xpos by\arrowlength
            \advance \ypos by\tempcounta
        \else
            \advance \xpos by-\arrowlength
            \advance \ypos by-\tempcounta
        \fi
    \fi
    \multiply \arrowtype by-1
    \multiply \rise by-1
    \multiply \run by-1
\fi
\ifnum \arrowtype=1
    \put(\xpos,\ypos){\vector(\run,\rise){\arrowlength}}%
\else\ifnum \arrowtype=2
    \put(\xpos,\ypos){\mvector(\run,\rise)\arrowlength}%
\else\ifnum\arrowtype=3
    \put(\xpos,\ypos){\evector(\run,\rise){\arrowlength}}%
\fi\fi\fi
}}
 
\def\putsplitvector(#1,#2)#3#4{%%
\xpos #1
\ypos #2
\arrowtype #4
\halflength #3
\arrowlength #3
\gap 140
\advance \halflength by-\gap
\divide \halflength by2
\ifnum \arrowtype=1
    \put(\xpos,\ypos){\line(0,-1){\halflength}}%
    \advance\ypos by-\halflength
    \advance\ypos by-\gap
    \put(\xpos,\ypos){\vector(0,-1){\halflength}}%
\else\ifnum \arrowtype=2
    \put(\xpos,\ypos){\line(0,-1)\halflength}%
    \put(\xpos,\ypos){\vector(0,-1)3}%
    \advance\ypos by-\halflength
    \advance\ypos by-\gap
    \put(\xpos,\ypos){\vector(0,-1){\halflength}}%
\else\ifnum\arrowtype=3
    \put(\xpos,\ypos){\line(0,-1)\halflength}%
    \advance\ypos by-\halflength
    \advance\ypos by-\gap
    \put(\xpos,\ypos){\evector(0,-1){\halflength}}%
\else\ifnum \arrowtype=-1
    \advance \ypos by-\arrowlength
    \put(\xpos,\ypos){\line(0,1){\halflength}}%
    \advance\ypos by\halflength
    \advance\ypos by\gap
    \put(\xpos,\ypos){\vector(0,1){\halflength}}%
\else\ifnum \arrowtype=-2
    \advance \ypos by-\arrowlength
    \put(\xpos,\ypos){\line(0,1)\halflength}%
    \put(\xpos,\ypos){\vector(0,1)3}%
    \advance\ypos by\halflength
    \advance\ypos by\gap
    \put(\xpos,\ypos){\vector(0,1){\halflength}}%
\else\ifnum\arrowtype=-3
    \advance \ypos by-\arrowlength
    \put(\xpos,\ypos){\line(0,1)\halflength}%
    \advance\ypos by\halflength
    \advance\ypos by\gap
    \put(\xpos,\ypos){\evector(0,1){\halflength}}%
\fi\fi\fi\fi\fi\fi
}
 
\def\putmorphism(#1)(#2,#3)[#4`#5`#6]#7#8#9{{%
\run #2
\rise #3
\ifnum\rise=0
  \puthmorphism(#1)[#4`#5`#6]{#7}{#8}{#9}%
\else\ifnum\run=0
  \putvmorphism(#1)[#4`#5`#6]{#7}{#8}{#9}%
\else
\setpos(#1)%
\arrowlength #7
\arrowtype #8
\ifnum\run=0
\else\ifnum\rise=0
\else
\ifnum\run>0
    \coefa=1
\else
   \coefa=-1
\fi
\ifnum\arrowtype>0
   \coefb=0
   \coefc=-1
\else
   \coefb=\coefa
   \coefc=1
   \arrowtype=-\arrowtype
\fi
\width=2
\multiply \width by\run
\divide \width by\rise
\ifnum \width<0  \width=-\width\fi
\advance\width by60
\if l#9 \width=-\width\fi
\putbox(\xpos,\ypos){#4}%            %node 1
{\multiply \coefa by\arrowlength%      %node 2
\advance\xpos by\coefa
\multiply \coefa by\rise
\divide \coefa by\run
\advance \ypos by\coefa
\putbox(\xpos,\ypos){#5} }%
{\multiply \coefa by\arrowlength%      %label
\divide \coefa by2
\advance \xpos by\coefa
\advance \xpos by\width
\multiply \coefa by\rise
\divide \coefa by\run
\advance \ypos by\coefa
\if l#9%
   \put(\xpos,\ypos){\makebox(0,0)[r]{$#6$}}%
\else\if r#9%
   \put(\xpos,\ypos){\makebox(0,0)[l]{$#6$}}%
\fi\fi }%
{\multiply \rise by-\coefc%             %arrow
\multiply \run by-\coefc
\multiply \coefb by\arrowlength
\advance \xpos by\coefb
\multiply \coefb by\rise
\divide \coefb by\run
\advance \ypos by\coefb
\multiply \coefc by70
\advance \ypos by\coefc
\multiply \coefc by\run
\divide \coefc by\rise
\advance \xpos by\coefc
\multiply \coefa by140
\multiply \coefa by\run
\divide \coefa by\rise
\advance \arrowlength by\coefa
\ifnum \arrowtype=1
   \put(\xpos,\ypos){\vector(\run,\rise){\arrowlength}}%
\else\ifnum\arrowtype=2
   \put(\xpos,\ypos){\mvector(\run,\rise){\arrowlength}}%
\else\ifnum\arrowtype=3
   \put(\xpos,\ypos){\evector(\run,\rise){\arrowlength}}%
\fi\fi\fi}\fi\fi\fi\fi}}
 
\def\puthmorphism(#1,#2)[#3`#4`#5]#6#7#8{{%
\xpos #1
\ypos #2
\width #6
\arrowlength #6
\putbox(\xpos,\ypos){#3\vphantom{#4}}%
{\advance \xpos by\arrowlength
\putbox(\xpos,\ypos){\vphantom{#3}#4}}%
\horsize{\tempcounta}{#3}%
\horsize{\tempcountb}{#4}%
\divide \tempcounta by2
\divide \tempcountb by2
\advance \tempcounta by30
\advance \tempcountb by30
\advance \xpos by\tempcounta
\advance \arrowlength by-\tempcounta
\advance \arrowlength by-\tempcountb
\putvector(\xpos,\ypos)(1,0){\arrowlength}{#7}%
\divide \arrowlength by2
\advance \xpos by\arrowlength
\vertsize{\tempcounta}{#5}%
\divide\tempcounta by2
\advance \tempcounta by20
\if a#8 %
   \advance \ypos by\tempcounta
   \putbox(\xpos,\ypos){#5}%
\else
   \advance \ypos by-\tempcounta
   \putbox(\xpos,\ypos){#5}%
\fi}}
 
\def\putvmorphism(#1,#2)[#3`#4`#5]#6#7#8{{%
\xpos #1
\ypos #2
\arrowlength #6
\arrowtype #7
\settowidth{\xlen}{$#5$}%
\putbox(\xpos,\ypos){#3}%
{\advance \ypos by-\arrowlength
\putbox(\xpos,\ypos){#4}}%
{\advance\arrowlength by-140
\advance \ypos by-70
\ifdim\xlen>0pt
   \if m#8%
      \putsplitvector(\xpos,\ypos){\arrowlength}{\arrowtype}%
   \else
      \putvector(\xpos,\ypos)(0,-1){\arrowlength}{\arrowtype}%
   \fi
\else
   \putvector(\xpos,\ypos)(0,-1){\arrowlength}{\arrowtype}%
\fi}%
\ifdim\xlen>0pt
   \divide \arrowlength by2
   \advance\ypos by-\arrowlength
   \if l#8%
      \advance \xpos by-40
      \put(\xpos,\ypos){\makebox(0,0)[r]{$#5$}}%
   \else\if r#8%
      \advance \xpos by40
      \put(\xpos,\ypos){\makebox(0,0)[l]{$#5$}}%
   \else
      \putbox(\xpos,\ypos){#5}%
   \fi\fi
\fi
}}
 
\def\topadjust[#1`#2`#3]{%
\yoff=10
\vertadjust[#1`#2`{#3}]%
\advance \yext by\tempcounta
\advance \yext by 10
}
\def\botadjust[#1`#2`#3]{%
\vertadjust[#1`#2`{#3}]%
\advance \yext by\tempcounta
\advance \yoff by-\tempcounta
}
\def\leftadjust[#1`#2`#3]{%
\xoff=0
\horadjust[#1`#2`{#3}]%
\advance \xext by\tempcounta
\advance \xoff by-\tempcounta
}
\def\rightadjust[#1`#2`#3]{%
\horadjust[#1`#2`{#3}]%
\advance \xext by\tempcounta
}
\def\rightsladjust[#1`#2`#3]{%
\sladjust[#1`#2`{#3}]{\width}%
\advance \xext by\tempcounta
}
\def\leftsladjust[#1`#2`#3]{%
\xoff=0
\sladjust[#1`#2`{#3}]{\width}%
\advance \xext by\tempcounta
\advance \xoff by-\tempcounta
}
\def\adjust[#1`#2;#3`#4;#5`#6;#7`#8]{%
\topadjust[#1``{#2}]
\leftadjust[#3``{#4}]
\rightadjust[#5``{#6}]
\botadjust[#7``{#8}]}
 
\def\putsquarep<#1>(#2)[#3;#4`#5`#6`#7]{{%
\setsqparms[#1]%
\setpos(#2)%
\settokens[#3]%
\puthmorphism(\xpos,\ypos)[\tokenc`\tokend`{#7}]{\width}{\arrowtyped}b%
\advance\ypos by \height
\puthmorphism(\xpos,\ypos)[\tokena`\tokenb`{#4}]{\width}{\arrowtypea}a%
\putvmorphism(\xpos,\ypos)[``{#5}]{\height}{\arrowtypeb}l%
\advance\xpos by \width
\putvmorphism(\xpos,\ypos)[``{#6}]{\height}{\arrowtypec}r%
}}
 
\def\putsquare{\@ifnextchar <{\putsquarep}{\putsquarep%
   <\arrowtypea`\arrowtypeb`\arrowtypec`\arrowtyped;\width`\height>}}
\def\square{\@ifnextchar< {\squarep}{\squarep
   <\arrowtypea`\arrowtypeb`\arrowtypec`\arrowtyped;\width`\height>}}
\def\squarep<#1>[#2`#3`#4`#5;#6`#7`#8`#9]{{%          %     #2------>#3
\setsqparms[#1]%                                      %      |       |
\xext=\width                                          %      |       |
\yext=\height                                         %    #7|       |#8
\topadjust[#2`#3`{#6}]%                               %      |       |
\botadjust[#4`#5`{#9}]%                               %      |       |
\leftadjust[#2`#4`{#7}]%                              %
\rightadjust[#3`#5`{#8}]%                             %     #4------>#5
\begin{picture}(\xext,\yext)(\xoff,\yoff)%                      #9
\putsquarep<\arrowtypea`\arrowtypeb`\arrowtypec`\arrowtyped;\width`\height>%
(0,0)[#2`#3`#4`#5;#6`#7`#8`{#9}]%
\end{picture}%
}}
 
\def\putptrianglep<#1>(#2,#3)[#4`#5`#6;#7`#8`#9]{{%
\settriparms[#1]%
\xpos=#2 \ypos=#3
\advance\ypos by \height
\puthmorphism(\xpos,\ypos)[#4`#5`{#7}]{\height}{\arrowtypea}a%
\putvmorphism(\xpos,\ypos)[`#6`{#8}]{\height}{\arrowtypeb}l%
\advance\xpos by\height
\putmorphism(\xpos,\ypos)(-1,-1)[``{#9}]{\height}{\arrowtypec}r%
}}
 
\def\putptriangle{\@ifnextchar <{\putptrianglep}{\putptrianglep
   <\arrowtypea`\arrowtypeb`\arrowtypec;\height>}}
\def\ptriangle{\@ifnextchar <{\ptrianglep}{\ptrianglep
   <\arrowtypea`\arrowtypeb`\arrowtypec;\height>}}
 
\def\ptrianglep<#1>[#2`#3`#4;#5`#6`#7]{{%%       #5
\settriparms[#1]%
\width=\height                         %      #2----->#3
\xext=\width                           %      |      /
\yext=\width                           %      |     /
\topadjust[#2`#3`{#5}]%                %    #6|    /#7
\botadjust[#3``]%                      %      |   /
\leftadjust[#2`#4`{#6}]%               %      |  /
\rightsladjust[#3`#4`{#7}]%            %
\begin{picture}(\xext,\yext)(\xoff,\yoff)%    #4
\putptrianglep<\arrowtypea`\arrowtypeb`\arrowtypec;\height>%
(0,0)[#2`#3`#4;#5`#6`{#7}]%
\end{picture}%
}}
 
\def\putqtrianglep<#1>(#2,#3)[#4`#5`#6;#7`#8`#9]{{%
\settriparms[#1]%
\xpos=#2 \ypos=#3
\advance\ypos by\height
\puthmorphism(\xpos,\ypos)[#4`#5`{#7}]{\height}{\arrowtypea}a%
\putmorphism(\xpos,\ypos)(1,-1)[``{#8}]{\height}{\arrowtypeb}l%
\advance\xpos by\height
\putvmorphism(\xpos,\ypos)[`#6`{#9}]{\height}{\arrowtypec}r%
}}
 
\def\putqtriangle{\@ifnextchar <{\putqtrianglep}{\putqtrianglep
   <\arrowtypea`\arrowtypeb`\arrowtypec;\height>}}
\def\qtriangle{\@ifnextchar <{\qtrianglep}{\qtrianglep
   <\arrowtypea`\arrowtypeb`\arrowtypec;\height>}}
 
\def\qtrianglep<#1>[#2`#3`#4;#5`#6`#7]{{%%
\settriparms[#1]%                                  #5
\width=\height                         %        #2----->#3
\xext=\width                           %         \      |
\yext=\height                          %          \     |
\topadjust[#2`#3`{#5}]%                %         #6\    |#7
\botadjust[#4``]%                      %            \   |
\leftsladjust[#2`#4`{#6}]%             %             \  |
\rightadjust[#3`#4`{#7}]%              %
\begin{picture}(\xext,\yext)(\xoff,\yoff)%             #4
\putqtrianglep<\arrowtypea`\arrowtypeb`\arrowtypec;\height>%
(0,0)[#2`#3`#4;#5`#6`{#7}]%
\end{picture}%
}}
 
\def\putdtrianglep<#1>(#2,#3)[#4`#5`#6;#7`#8`#9]{{%
\settriparms[#1]%
\xpos=#2 \ypos=#3
\puthmorphism(\xpos,\ypos)[#5`#6`{#9}]{\height}{\arrowtypec}b%
\advance\xpos by \height \advance\ypos by\height
\putmorphism(\xpos,\ypos)(-1,-1)[``{#7}]{\height}{\arrowtypea}l%
\putvmorphism(\xpos,\ypos)[#4``{#8}]{\height}{\arrowtypeb}r%
}}
 
\def\putdtriangle{\@ifnextchar <{\putdtrianglep}{\putdtrianglep
   <\arrowtypea`\arrowtypeb`\arrowtypec;\height>}}
\def\dtriangle{\@ifnextchar <{\dtrianglep}{\dtrianglep
   <\arrowtypea`\arrowtypeb`\arrowtypec;\height>}}
 
\def\dtrianglep<#1>[#2`#3`#4;#5`#6`#7]{{%%
\settriparms[#1]%                                          #2
\width=\height                         %                  / |
\xext=\width                           %                 /  |
\yext=\height                          %              #5/   |#6
\topadjust[#2``]%                      %               /    |
\botadjust[#3`#4`{#7}]%                %              /     |
\leftsladjust[#3`#2`{#5}]%             %
\rightadjust[#2`#4`{#6}]%              %            #3----->#4
\begin{picture}(\xext,\yext)(\xoff,\yoff)%              #7
\putdtrianglep<\arrowtypea`\arrowtypeb`\arrowtypec;\height>%
(0,0)[#2`#3`#4;#5`#6`{#7}]%
\end{picture}%
}}
 
\def\putbtrianglep<#1>(#2,#3)[#4`#5`#6;#7`#8`#9]{{%
\settriparms[#1]%
\xpos=#2 \ypos=#3
\puthmorphism(\xpos,\ypos)[#5`#6`{#9}]{\height}{\arrowtypec}b%
\advance\ypos by\height
\putmorphism(\xpos,\ypos)(1,-1)[``{#8}]{\height}{\arrowtypeb}r%
\putvmorphism(\xpos,\ypos)[#4``{#7}]{\height}{\arrowtypea}l%
}}
 
\def\putbtriangle{\@ifnextchar <{\putbtrianglep}{\putbtrianglep
   <\arrowtypea`\arrowtypeb`\arrowtypec;\height>}}
\def\btriangle{\@ifnextchar <{\btrianglep}{\btrianglep
   <\arrowtypea`\arrowtypeb`\arrowtypec;\height>}}
 
\def\btrianglep<#1>[#2`#3`#4;#5`#6`#7]{{%%
\settriparms[#1]%                                     #2
\width=\height                         %              | \
\xext=\width                           %              |  \
\yext=\height                          %            #5|   \#6
\topadjust[#2``]%                      %              |    \
\botadjust[#3`#4`{#7}]%                %              |     \
\leftadjust[#2`#3`{#5}]%               %
\rightsladjust[#4`#2`{#6}]%            %              #3----->#4
\begin{picture}(\xext,\yext)(\xoff,\yoff)%                #7
\putbtrianglep<\arrowtypea`\arrowtypeb`\arrowtypec;\height>%
(0,0)[#2`#3`#4;#5`#6`{#7}]%
\end{picture}%
}}
 
\def\putAtrianglep<#1>(#2,#3)[#4`#5`#6;#7`#8`#9]{{%
\settriparms[#1]%
\xpos=#2 \ypos=#3
{\multiply \height by2
\puthmorphism(\xpos,\ypos)[#5`#6`{#9}]{\height}{\arrowtypec}b}%
\advance\xpos by\height \advance\ypos by\height
\putmorphism(\xpos,\ypos)(-1,-1)[#4``{#7}]{\height}{\arrowtypea}l%
\putmorphism(\xpos,\ypos)(1,-1)[``{#8}]{\height}{\arrowtypeb}r%
}}
 
\def\putAtriangle{\@ifnextchar <{\putAtrianglep}{\putAtrianglep
   <\arrowtypea`\arrowtypeb`\arrowtypec;\height>}}
\def\Atriangle{\@ifnextchar <{\Atrianglep}{\Atrianglep
   <\arrowtypea`\arrowtypeb`\arrowtypec;\height>}}
 
\def\Atrianglep<#1>[#2`#3`#4;#5`#6`#7]{{%%
\settriparms[#1]%                                 #2
\width=\height                         %         /   \
\xext=\width                           %        /     \
\yext=\height                          %     #5/       \#6
\topadjust[#2``]%                      %      /         \
\botadjust[#3`#4`{#7}]%                %     /           \
\multiply \xext by2 %                  %
\leftsladjust[#3`#2`{#5}]%             %   #3------------>#4
\rightsladjust[#4`#2`{#6}]%            %          #7
\begin{picture}(\xext,\yext)(\xoff,\yoff)%
\putAtrianglep<\arrowtypea`\arrowtypeb`\arrowtypec;\height>%
(0,0)[#2`#3`#4;#5`#6`{#7}]%
\end{picture}%
}}
 
\def\putAtrianglepairp<#1>(#2)[#3;#4`#5`#6`#7`#8]{{
\settripairparms[#1]%
\setpos(#2)%
\settokens[#3]%
\puthmorphism(\xpos,\ypos)[\tokenb`\tokenc`{#7}]{\height}{\arrowtyped}b%
\advance\xpos by\height
\advance\ypos by\height
\putmorphism(\xpos,\ypos)(-1,-1)[\tokena``{#4}]{\height}{\arrowtypea}l%
\putvmorphism(\xpos,\ypos)[``{#5}]{\height}{\arrowtypeb}m%
\putmorphism(\xpos,\ypos)(1,-1)[``{#6}]{\height}{\arrowtypec}r%
}}
 
\def\putAtrianglepair{\@ifnextchar <{\putAtrianglepairp}{\putAtrianglepairp%
   <\arrowtypea`\arrowtypeb`\arrowtypec`\arrowtyped`\arrowtypee;\height>}}
\def\Atrianglepair{\@ifnextchar <{\Atrianglepairp}{\Atrianglepairp%
   <\arrowtypea`\arrowtypeb`\arrowtypec`\arrowtyped`\arrowtypee;\height>}}
 
\def\Atrianglepairp<#1>[#2;#3`#4`#5`#6`#7]{{%
\settripairparms[#1]%
\settokens[#2]%
\width=\height
\xext=\width
\yext=\height
\topadjust[\tokena``]%
\vertadjust[\tokenb`\tokenc`{#6}]%                      %  #2a
\tempcountd=\tempcounta                       %           / | \
\vertadjust[\tokenc`\tokend`{#7}]%            %          /  |  \
\ifnum\tempcounta<\tempcountd                 %       #3/  #4   \#5
\tempcounta=\tempcountd\fi                    %        /    |    \
\advance \yext by\tempcounta                  %       /     |     \
\advance \yoff by-\tempcounta                 %
\multiply \xext by2 %                         %     #2b---->#2c---->#2d
\leftsladjust[\tokenb`\tokena`{#3}]%          %         #6     #7
\rightsladjust[\tokend`\tokena`{#5}]%
\begin{picture}(\xext,\yext)(\xoff,\yoff)%
\putAtrianglepairp
<\arrowtypea`\arrowtypeb`\arrowtypec`\arrowtyped`\arrowtypee;\height>%
(0,0)[#2;#3`#4`#5`#6`{#7}]%
\end{picture}%
}}
 
\def\putVtrianglep<#1>(#2,#3)[#4`#5`#6;#7`#8`#9]{{%
\settriparms[#1]%
\xpos=#2 \ypos=#3
\advance\ypos by\height
{\multiply\height by2
\puthmorphism(\xpos,\ypos)[#4`#5`{#7}]{\height}{\arrowtypea}a}%
\putmorphism(\xpos,\ypos)(1,-1)[`#6`{#8}]{\height}{\arrowtypeb}l%
\advance\xpos by\height
\advance\xpos by\height
\putmorphism(\xpos,\ypos)(-1,-1)[``{#9}]{\height}{\arrowtypec}r%
}}

\def\putVtriangle{\@ifnextchar <{\putVtrianglep}{\putVtrianglep
   <\arrowtypea`\arrowtypeb`\arrowtypec;\height>}}
\def\Vtriangle{\@ifnextchar <{\Vtrianglep}{\Vtrianglep
   <\arrowtypea`\arrowtypeb`\arrowtypec;\height>}}
 
\def\Vtrianglep<#1>[#2`#3`#4;#5`#6`#7]{{%%
\settriparms[#1]%                                      #5
\width=\height                         %        #2------------->#3
\xext=\width                           %         \             /
\yext=\height                          %          \           /
\topadjust[#2`#3`{#5}]%                %         #6\         /#7
\botadjust[#4``]%                      %            \       /
\multiply \xext by2 %                  %             \     /
\leftsladjust[#2`#3`{#6}]%             %
\rightsladjust[#3`#4`{#7}]%            %               #4
\begin{picture}(\xext,\yext)(\xoff,\yoff)%
\putVtrianglep<\arrowtypea`\arrowtypeb`\arrowtypec;\height>%
(0,0)[#2`#3`#4;#5`#6`{#7}]%
\end{picture}%
}}
 
\def\putVtrianglepairp<#1>(#2)[#3;#4`#5`#6`#7`#8]{{
\settripairparms[#1]%
\setpos(#2)%
\settokens[#3]%
\advance\ypos by\height
\putmorphism(\xpos,\ypos)(1,-1)[`\tokend`{#6}]{\height}{\arrowtypec}l%
\puthmorphism(\xpos,\ypos)[\tokena`\tokenb`{#4}]{\height}{\arrowtypea}a%
\advance\xpos by\height
\putvmorphism(\xpos,\ypos)[``{#7}]{\height}{\arrowtyped}m%
\advance\xpos by\height
\putmorphism(\xpos,\ypos)(-1,-1)[``{#8}]{\height}{\arrowtypee}r%
}}
 
\def\putVtrianglepair{\@ifnextchar <{\putVtrianglepairp}{\putVtrianglepairp%
    <\arrowtypea`\arrowtypeb`\arrowtypec`\arrowtyped`\arrowtypee;\height>}}
\def\Vtrianglepair{\@ifnextchar <{\Vtrianglepairp}{\Vtrianglepairp%
    <\arrowtypea`\arrowtypeb`\arrowtypec`\arrowtyped`\arrowtypee;\height>}}
 
\def\Vtrianglepairp<#1>[#2;#3`#4`#5`#6`#7]{{%
\settripairparms[#1]%
\settokens[#2]%                            #3      #4
\xext=\height                  %        #2a---->#2b---->#2c
\width=\height                 %         \      |      /
\yext=\height                  %          \     |     /
\vertadjust[\tokena`\tokenb`{#4}]%       #5\   #6    /#7
\tempcountd=\tempcounta        %            \   |   /
\vertadjust[\tokenb`\tokenc`{#5}]%           \  |  /
\ifnum\tempcounta<\tempcountd%
\tempcounta=\tempcountd\fi%                    #2d
\advance \yext by\tempcounta
\botadjust[\tokend``]%
\multiply \xext by2
\leftsladjust[\tokena`\tokend`{#6}]%
\rightsladjust[\tokenc`\tokend`{#7}]%
\begin{picture}(\xext,\yext)(\xoff,\yoff)%
\putVtrianglepairp
<\arrowtypea`\arrowtypeb`\arrowtypec`\arrowtyped`\arrowtypee;\height>%
(0,0)[#2;#3`#4`#5`#6`{#7}]%
\end{picture}%
}}

\def\putCtrianglep<#1>(#2,#3)[#4`#5`#6;#7`#8`#9]{{%
\settriparms[#1]%
\xpos=#2 \ypos=#3
\advance\ypos by\height
\putmorphism(\xpos,\ypos)(1,-1)[``{#9}]{\height}{\arrowtypec}l%
\advance\xpos by\height
\advance\ypos by\height
\putmorphism(\xpos,\ypos)(-1,-1)[#4`#5`{#7}]{\height}{\arrowtypea}l%
{\multiply\height by 2
\putvmorphism(\xpos,\ypos)[`#6`{#8}]{\height}{\arrowtypeb}r}%
}}
 
\def\putCtriangle{\@ifnextchar <{\putCtrianglep}{\putCtrianglep
    <\arrowtypea`\arrowtypeb`\arrowtypec;\height>}}
\def\Ctriangle{\@ifnextchar <{\Ctrianglep}{\Ctrianglep
    <\arrowtypea`\arrowtypeb`\arrowtypec;\height>}}
 
\def\Ctrianglep<#1>[#2`#3`#4;#5`#6`#7]{{%%
\settriparms[#1]%                                         #2
\width=\height                          %                / |
\xext=\width                            %               /  |
\yext=\height                           %            #5/   |
\multiply \yext by2 %                   %             /    |
\topadjust[#2``]%                       %            /     |
\botadjust[#4``]%                       %           v      |
\sladjust[#3`#2`{#5}]{\width}%          %          #3      |#6
\tempcountd=\tempcounta                 %           \      |
\sladjust[#3`#4`{#7}]{\width}%          %            \     |
\ifnum \tempcounta<\tempcountd          %           #7\    |
\tempcounta=\tempcountd\fi              %              \   |
\advance \xext by\tempcounta            %               \  |
\advance \xoff by-\tempcounta           %
\rightadjust[#2`#4`{#6}]%               %                 #4
\begin{picture}(\xext,\yext)(\xoff,\yoff)%
\putCtrianglep<\arrowtypea`\arrowtypeb`\arrowtypec;\height>%
(0,0)[#2`#3`#4;#5`#6`{#7}]%
\end{picture}%
}}
 
\def\putDtrianglep<#1>(#2,#3)[#4`#5`#6;#7`#8`#9]{{%
\settriparms[#1]%
\xpos=#2 \ypos=#3
\advance\xpos by\height \advance\ypos by\height
\putmorphism(\xpos,\ypos)(-1,-1)[``{#9}]{\height}{\arrowtypec}r%
\advance\xpos by-\height \advance\ypos by\height
\putmorphism(\xpos,\ypos)(1,-1)[`#5`{#8}]{\height}{\arrowtypeb}r%
{\multiply\height by 2
\putvmorphism(\xpos,\ypos)[#4`#6`{#7}]{\height}{\arrowtypea}l}%
}}
 
\def\putDtriangle{\@ifnextchar <{\putDtrianglep}{\putDtrianglep
    <\arrowtypea`\arrowtypeb`\arrowtypec;\height>}}
\def\Dtriangle{\@ifnextchar <{\Dtrianglep}{\Dtrianglep
   <\arrowtypea`\arrowtypeb`\arrowtypec;\height>}}
 
\def\Dtrianglep<#1>[#2`#3`#4;#5`#6`#7]{{%%
\settriparms[#1]%                                 #2
\width=\height                         %          | \
\xext=\height                          %          |  \
\yext=\height                          %          |   \#6
\multiply \yext by2 %                  %          |    \
\topadjust[#2``]%                      %          |     \
\botadjust[#4``]%                      %          |
\leftadjust[#2`#4`{#5}]%               %        #5|      #3
\sladjust[#3`#2`{#5}]{\height}%        %          |      /
\tempcountd=\tempcountd                %          |     /
\sladjust[#3`#4`{#7}]{\height}%        %          |    /#7
\ifnum \tempcounta<\tempcountd         %          |   /
\tempcounta=\tempcountd\fi             %          |  /
\advance \xext by\tempcounta           %
\begin{picture}(\xext,\yext)(\xoff,\yoff)%        #4
\putDtrianglep<\arrowtypea`\arrowtypeb`\arrowtypec;\height>%
(0,0)[#2`#3`#4;#5`#6`{#7}]%
\end{picture}%
}}
 
\def\setrecparms[#1`#2]{\width=#1 \height=#2}%
\def\recursep<#1`#2>[#3;#4`#5`#6`#7`#8]{{%
\width=#1 \height=#2
\settokens[#3]
\settowidth{\tempdimen}{$\tokena$}
\ifdim\tempdimen=0pt
  \savebox{\tempboxa}{\hbox{$\tokenb$}}%
  \savebox{\tempboxb}{\hbox{$\tokend$}}%
  \savebox{\tempboxc}{\hbox{$#6$}}%
\else
  \savebox{\tempboxa}{\hbox{$\hbox{$\tokena$}\times\hbox{$\tokenb$}$}}%
  \savebox{\tempboxb}{\hbox{$\hbox{$\tokena$}\times\hbox{$\tokend$}$}}%
  \savebox{\tempboxc}{\hbox{$\hbox{$\tokena$}\times\hbox{$#6$}$}}%
\fi
\ypos=\height
\divide\ypos by 2
\xpos=\ypos
\advance\xpos by \width
\xext=\xpos \yext=\height
\topadjust[#3`\usebox{\tempboxa}`{#4}]%
\botadjust[#5`\usebox{\tempboxb}`{#8}]%
\sladjust[\tokenc`\tokenb`{#5}]{\ypos}%
\tempcountd=\tempcounta
\sladjust[\tokenc`\tokend`{#5}]{\ypos}%
\ifnum \tempcounta<\tempcountd
\tempcounta=\tempcountd\fi
\advance \xext by\tempcounta
\advance \xoff by-\tempcounta
\rightadjust[\usebox{\tempboxa}`\usebox{\tempboxb}`\usebox{\tempboxc}]%
\bfig
\putCtrianglep<-1`1`1;\ypos>(0,0)[`\tokenc`;#5`#6`{#7}]%
\puthmorphism(\ypos,0)[\tokend`\usebox{\tempboxb}`{#8}]{\width}{-1}b%
\puthmorphism(\ypos,\height)[\tokenb`\usebox{\tempboxa}`{#4}]{\width}{-1}a%
\advance\ypos by \width
\putvmorphism(\ypos,\height)[``\usebox{\tempboxc}]{\height}1r%
\efig
}}
 
\def\recurse{\@ifnextchar <{\recursep}{\recursep<\width`\height>}}
 
\def\puttwohmorphisms(#1,#2)[#3`#4;#5`#6]#7#8#9{{%
\puthmorphism(#1,#2)[#3`#4`]{#7}0a
\ypos=#2
\advance\ypos by 20
\puthmorphism(#1,\ypos)[\phantom{#3}`\phantom{#4}`#5]{#7}{#8}a
\advance\ypos by -40
\puthmorphism(#1,\ypos)[\phantom{#3}`\phantom{#4}`#6]{#7}{#9}b
}}
 
\def\puttwovmorphisms(#1,#2)[#3`#4;#5`#6]#7#8#9{{%
\putvmorphism(#1,#2)[#3`#4`]{#7}0a
\xpos=#1
\advance\xpos by -20
\putvmorphism(\xpos,#2)[\phantom{#3}`\phantom{#4}`#5]{#7}{#8}l
\advance\xpos by 40
\putvmorphism(\xpos,#2)[\phantom{#3}`\phantom{#4}`#6]{#7}{#9}r
}}
 
\def\puthcoequalizer(#1)[#2`#3`#4;#5`#6`#7]#8#9{{%
\setpos(#1)%
\puttwohmorphisms(\xpos,\ypos)[#2`#3;#5`#6]{#8}11%
\advance\xpos by #8
\puthmorphism(\xpos,\ypos)[\phantom{#3}`#4`#7]{#8}1{#9}
}}
 
\def\putvcoequalizer(#1)[#2`#3`#4;#5`#6`#7]#8#9{{%
\setpos(#1)%
\puttwovmorphisms(\xpos,\ypos)[#2`#3;#5`#6]{#8}11%
\advance\ypos by -#8
\putvmorphism(\xpos,\ypos)[\phantom{#3}`#4`#7]{#8}1{#9}
}}
 
\def\putthreehmorphisms(#1)[#2`#3;#4`#5`#6]#7(#8)#9{{%
\setpos(#1) \settypes(#8)
\if a#9 %
     \vertsize{\tempcounta}{#5}%
     \vertsize{\tempcountb}{#6}%
     \ifnum \tempcounta<\tempcountb \tempcounta=\tempcountb \fi
\else
     \vertsize{\tempcounta}{#4}%
     \vertsize{\tempcountb}{#5}%
     \ifnum \tempcounta<\tempcountb \tempcounta=\tempcountb \fi
\fi
\advance \tempcounta by 60
\puthmorphism(\xpos,\ypos)[#2`#3`#5]{#7}{\arrowtypeb}{#9}
\advance\ypos by \tempcounta
\puthmorphism(\xpos,\ypos)[\phantom{#2}`\phantom{#3}`#4]{#7}{\arrowtypea}{#9}
\advance\ypos by -\tempcounta \advance\ypos by -\tempcounta
\puthmorphism(\xpos,\ypos)[\phantom{#2}`\phantom{#3}`#6]{#7}{\arrowtypec}{#9}
}}
 
\def\putarc(#1,#2)[#3`#4`#5]#6#7#8{{%
\xpos #1
\ypos #2
\width #6
\arrowlength #6
\putbox(\xpos,\ypos){#3\vphantom{#4}}%
{\advance \xpos by\arrowlength
\putbox(\xpos,\ypos){\vphantom{#3}#4}}%
\horsize{\tempcounta}{#3}%
\horsize{\tempcountb}{#4}%
\divide \tempcounta by2
\divide \tempcountb by2
\advance \tempcounta by30
\advance \tempcountb by30
\advance \xpos by\tempcounta
\advance \arrowlength by-\tempcounta
\advance \arrowlength by-\tempcountb
\halflength=\arrowlength \divide\halflength by 2
\divide\arrowlength by 5
\put(\xpos,\ypos){\bezier{\arrowlength}(0,0)(50,50)(\halflength,50)}
\ifnum #7=-1 \put(\xpos,\ypos){\vector(-3,-2)0} \fi
\advance\xpos by \halflength
\put(\xpos,\ypos){\xpos=\halflength \advance\xpos by -50
   \bezier{\arrowlength}(0,50)(\xpos,50)(\halflength,0)}
\ifnum #7=1 {\advance \xpos by
   \halflength \put(\xpos,\ypos){\vector(3,-2)0}} \fi
\advance\ypos by 50
\vertsize{\tempcounta}{#5}%
\divide\tempcounta by2
\advance \tempcounta by20
\if a#8 %
   \advance \ypos by\tempcounta
   \putbox(\xpos,\ypos){#5}%
\else
   \advance \ypos by-\tempcounta
   \putbox(\xpos,\ypos){#5}%
\fi
}}
 
\makeatother

\newtheorem{theorem}{Satz}

\newtheorem{corollary}[theorem]{Folgerung}

\newtheorem{beispiel}{Beispiel}

\newcommand{\moo}{{M^{\circ\circ}}}

\newcommand{\py}{{\mathfrak{p}}}
\newcommand{\my}{{\mathfrak{m}}}
\newcommand{\qy}{{\mathfrak{q}}}

\newcommand{\cC}{{\mathcal {C}}}
\newcommand{\cL}{{\mathcal {L}}}

\newenvironment{proof}{{\it Beweis}. $\;\;$}{\hspace*{\fill} $\Box$}

\begin{document}

\title{\Large\bf Eine Charakterisierung der Matlis-reflexiven Moduln}

\author{Helmut Z"oschinger\\ 
Mathematisches Institut der Universit"at M"unchen\\
Theresienstr. 39, D-80333 M"unchen, Germany\\
E-mail: zoeschinger@mathematik.uni-muenchen.de}

\date{}
\maketitle

\vspace{1cm}

\begin{center}  
{\bf Abstract}
\end{center}

Let $(R,\my)$ be a noetherian local ring, $E$ the injective hull of $k=R/\my$ and $M^\circ=$ Hom$_R(M,E)$ the Matlis dual of the $R$-module $M$. If the canonical monomorphism $\varphi: M \to \moo$ is surjective, $M$ is known to be called (Matlis-)reflexive. With the help of the Bass numbers $\mu(\py,M)=\dim_{\kappa(\py)}($Hom$_R(R/\py,M)_\py)$ of $M$ with respect to $\py$ we show: $M$ is reflexive if and only if $\mu(\py,M)=\mu(\py,\moo)$  for all $\py \in $ Spec$(R)$.
From this it follows for every $R$-module $M$: If there exists a monomorphism $\moo \hookrightarrow M$ or an epimorphism $M \twoheadrightarrow \moo$, then $M$ is already reflexive.

\vspace{0.5cm}


\vspace{0.5cm}


%\vspace{0.5cm}

%\renewcommand{\thesection}{\Roman{section}}
\section{Der Rang eines Moduls}

Stets sei $R$ noethersch und lokal. F"ur jeden $R$-Modul $M$ hei"st $\mu^i(\py,M):=\dim_{\kappa(\py)}($Ext$_R^i(R/\py,\\M)_\py)\;\;(i\geq 0,\;\py \in $ Spec$(R))$ die  $i$-te {\it Bass-Zahl} von $M$ bez"uglich $\py\;$ (siehe \cite[p.200]{1}), und weil wir sie im folgenden nur f"ur $i=0$ brauchen, schreiben wir statt $\mu^0(\py,M)$ kurz $\mu(\py,M)$.
Ist $R$ ein Integrit"atsring mit Quotientenk"orper $K$, so ist $\mu(0,M)=\dim_K(K\otimes_R M)$ offenbar der {\it Rang} von $M$.
Ist $R$ beliebig, wird $M[\py]:=$ Ann$_M(\py)$ zu einem Modul "uber dem Integrit"atsring $R/\py$ und es folgt $\mu(\py,M)=$ Rang$_{R/\py}(M[\py])$.
Insbesondere ist $\mu(\py,M)=0$ "aquivalent mit $\py \notin $ Ass$(M)$.
Allgemeiner gilt mit einem gro"sen Untermodul $U$ von $M,\;\,U\cong \coprod (R/\qy)^{(I_\qy)}(\qy \in $ Spec$(R)),$ da"s $\mu(\py,M)=\mu(\py,U)=\,|I_\py|$ ist.

\begin{theorem}
Ist $M$ ein $R$-Modul und $\varphi:M \to \moo$ die kanonische Einbettung, so gilt f"ur jedes Primideal $\py$:
$$\mu(\py,M)\;=\;\mu(\py,\moo)\;\Longleftrightarrow\; \py \notin \mbox{ Ass}(\mbox{Kok }  \varphi).$$
\end{theorem}

\begin{proof}
{\it 1.Schritt}\quad
Sei $R$ ein Integrit"atsring und $M$ ein $R$-Modul mit Rang$(M)= $ Rang$(\moo)$.
Dann ist Rang$(M)$ endlich. Zum Beweis w"ahle man einen
freien Untermodul $U$ von $M$ mit $M/U$ torsion, so da"s aus Rang$(U)=$ Rang$(M)$ auch Rang$(U)=$ Rang$(U^{\circ\circ})$ folgt.
Mit $U\cong R^{(I)}$ ist $E^{(I)} \to U^\circ$, also auch $\hat{R}^I \to U^{\circ\circ}$ ein zerfallender Monomorphismus, insbesondere Rang$(R^{(I)})=$ Rang$(R^I)$.

Daraus folgt die Endlichkeit von $I$: Mit $S = R^{(I)}$ und $P=R^I$ ist auch $P/S$ flach, so da"s es nach \cite[Theorem 7.10]{2} einen freien Zwischenmodul $S\subset F \subset P$ gibt mit $F$ rein in $P,\;\,F+\my P=P.$
Klar ist Rang$(S)=\dim_k(k^{(I)})$, andererseits Rang$(F)=\dim_k(F/\my F)\,= \dim_k(P/\my P)\,=\dim_k(k^I)$, so da"s aus Rang$(S)=$ Rang$(F)$ mit Erd"os-Kaplansky die Behauptung folgt.

{\it 2.Schritt} \quad
Sei $R$ ein Integrit"atsring und $M$ ein beliebiger $R$-Modul. Dann gilt:
$$\mbox{Rang}(M)\;=\;\mbox{Rang}(\moo)\;\Longleftrightarrow\; \mbox{Kok } \varphi \mbox{ ist torsion}.$$
Der Beweis folgt unmittelbar aus der Gleichung Rang$(\moo)=$ Rang$(M)+$ Rang(Kok $\varphi)$, denn bei ''$\Rightarrow$'' sind nach dem ersten Schritt alle Kardinalzahlen endlich, also Rang(Kok $ \varphi)=0,$
und ''$\Leftarrow$'' ist klar.

{\it 3.Schritt} \quad Sind jetzt $R$ und $M$ beliebig, gibt es zu jedem Primideal $\py$ nach \cite[Lemma 1.4]{5} ein kommutatives Diagramm

\vspace{0.5cm}

$$ 0\;\longrightarrow\;M[\py]\;\longrightarrow \;M[\py]^{\Box\,\Box}\;\longrightarrow\;\mbox{Kok} (\varphi_{M[\py]})\;\longrightarrow \;0$$
%\vspace{0.2cm}

\hspace{4.2cm} $\parallel$ \hspace{1.6cm} $\mu\;\downarrow$\hspace{2.0cm} $\nu\;\downarrow$
%\vspace{0.2cm}
$$0\;\longrightarrow\;M[\py]\;\longrightarrow\;(\moo)[\py]\;\longrightarrow \;(\mbox{Kok }\varphi)[\py]\;\longrightarrow \;0$$

\vspace{0.5cm}

\noindent{mit exakten Zeilen, in dem $\mu$, also auch $\nu$ ein Isomorphismus ist (wobei $\Box$ das Matlis-Duale "uber dem lokalen Ring $R/\py$ sei). Genau dann ist jetzt $\mu(\py,M)=\mu(\py,\moo)$, wenn $M[\py]$ und $M[\py]^{\Box\,\Box}$ denselben Rang "uber $R/\py$ haben, d.h. nach dem zweiten Schritt Kok$(\varphi_{M[\py]})$ "uber $R/\py$ torsion ist, d.h. (Kok $\varphi)[\py]$ "uber $R/\py$ torsion ist, d.h. $\py \notin $ Ass(Kok $\varphi).$}
\end{proof}

\begin{corollary}
Genau dann ist $M$ reflexiv, wenn $\mu(\py,M)=\mu(\py,\moo)$ gilt f"ur alle Primideale $\py$.
\end{corollary}

\begin{corollary}
Gibt es einen Monomorphismus $\moo \hookrightarrow M$ oder einen Epimorphismus $M\twoheadrightarrow \moo$, so ist $M$ bereits reflexiv.
\end{corollary}

\begin{proof}
Bei der ersten Folgerung ist Kok $\varphi =0$ "aquivalent mit der rechten Seite. Bei der zweiten gilt f"ur jeden Monomorphismus $f: X \hookrightarrow M$, da"s $\mu(\py,X)\leq \mu(\py,M)$ ist f"ur alle $\py$, bei $X=\moo$ also 
$\mu(\py,\moo)=\mu(\py,M)$ wie gew"unscht.
Ist aber $g: M \twoheadrightarrow \moo$ ein Epimorphismus, wird $g^\circ: M^{\circ\circ\circ} \to M^\circ$ ein Monomorphismus, so da"s nach eben $M^\circ$ reflexiv ist, also auch $M$.
\end{proof}

\noindent{{\bf Bemerkung 1.4}\quad
In einigen Spezialf"allen l"a"st sich sofort entscheiden, wann ein Primideal $\py$ die Bedingung im Satz erf"ullt:\\}
(1) F"ur alle $\py \notin $ Ass$(\moo)$ gilt $\mu(\py,M)=\mu(\py,\moo)\;$ (denn beide Seiten sind Null).\\
(2) Ist $M=E^n$ mit $n\geq 1$, gilt $\mu(\py,M)=\mu(\py,\moo)$ genau dann, wenn $R/\py$ vollst"andig ist (denn die Reflexivit"at von $M[\my]$ liefert Ass(Kok $\varphi)= $ Ass$(\moo)\backslash \{\my\}$, und bekanntlich ist Ass($\moo)
= $ Koass$(\hat{R})\,=\{\my\} \stackrel{\cdot}{\cup} \{\py \in $ Spec$(R) \,|\,R/\py$ ist unvollst"andig\}).\\
(3) Ist $M=E^{(I)}$ mit $I$ unendlich, gibt es "uberhaupt kein Primideal $\py$ mit $\mu(\py,M)=\mu(\py,\moo)\;$ (denn auch $M[\py]$ ist nicht reflexiv, also nach Beispiel 1 im 2.Abschnitt Rang$(M[\py])\\<$ Rang$(M[\py]^{\Box\,\Box})$ "uber $R/\py$).

\section{Starke (Ko--)Torsionsmoduln}
\setcounter{theorem}{0}

Ist $R$ ein Integrit"atsring, so ist die Klasse $\cC$ aller $R$-Moduln $M$ mit Rang$(M)=$ Rang$(\moo)$ nach dem 2.Schritt in (1.1) gegen"uber Untermoduln, Faktormoduln und Gruppenerweiterungen abgeschlossen.
Ein torsionsfreier $R$-Modul $M$ geh"ort genau dann zu $\cC$, wenn Kok $\varphi=0$, d.h. $M$ reflexiv ist ---
siehe die genauere Beschreibung in (2.3).
Ein Torsionsmodul $M$ geh"ort genau dann zu $\cC$, wenn $\moo$ torsion ist --- und daf"ur haben wir keine explizite Beschreibung.
Nur wenn $\moo$ sogar {\it stark torsion}, d.h. $\bigcap $ Ass$(\moo)\neq 0$ ist (siehe \cite[p.126]{5}, k"onnen wir die Struktur von $M$ angeben:

\begin{theorem}
Ist $R$ ein Integrit"atsring, kein K"orper, so sind f"ur einen $R$-Modul $M$ "aquivalent:
\begin{enumerate}
\item[(i)]
$\moo$ ist stark torsion.
\item[(ii)]
$D(M)$ ist artinsch und reflexiv, $M/D(M)$ ist beschr"ankt.
\end{enumerate}
\end{theorem}

\begin{proof} $(i \to ii)$\quad Wegen Ass$(\moo)= $ Koass$(M^\circ)$ ist nach Voraussetzung\\ $\bigcap $ Koass$(M^\circ)\neq 0$, d.h. $M^\circ$ im Sinne von \cite[p.126]{5} {\it stark kotorsion}, so da"s nach dem dortigen Satz 2.3 in der exakten Folge
$$0\;\longrightarrow\; T(M^\circ)\;\subset\; M^\circ\;\longrightarrow \; M^\circ/T(M^\circ)\;\longrightarrow \;0$$
das erste Glied beschr"ankt und das dritte endlich erzeugt ist.
F"ur den divisiblen Anteil $D(M)$ ist dann $D(M)^\circ$ torsionsfrei, als Faktormodul von $M^\circ/T(M^\circ)$ also
endlich erzeugt und reflexiv (sogar rein-injektiv), also $D(M)$ selbst artinsch und reflexiv.
In der exakten Folge
$$0\;\longrightarrow\; [M^\circ/T(M^\circ)]^\circ \;\longrightarrow\; \moo\;\longrightarrow \;[T(M^\circ)]^\circ\;\longrightarrow \;0$$
ist dann das erste Glied artinsch und das dritte beschr"ankt, also $r \cdot \moo$ artinsch f"ur ein $0\neq r\in R.$
Es folgt $r \cdot M$ artinsch, $r \cdot M/D(r \cdot M)$  beschr"ankt, also auch $M/D(M)$ beschr"ankt wie behauptet.

$(ii \to i)$\quad Zu $D=D(M)$ gibt es nach Voraussetzung ein $0\neq r\in \my$ mit $r \cdot M/D =0$. F"ur jedes $\py \in $ Ass$(\moo)$ gilt dann $\py \in $ Ass$(D^{\circ\circ})$ oder
$\py \in $ Ass$((M/D)^{\circ\circ}),\;\, \py =\my$ oder $r \in \py$, also $r \in \bigcap $ Ass$(\moo)$ wie verlangt.
\end{proof}

\noindent{{\bf Bemerkung 2.2}\quad
Ist $R$ ein Integrit"atsring und $\moo$ nur torsion, so k"onnen wir wenigstens zeigen:\\}
(1) Stets ist $D(M)$ artinsch.\\
(2) Ist $R$ ein unvollst"andiger Nagataring, so ist $M$ sogar beschr"ankt.\\
(3) Besitzt Koass$(M^\circ)$ eine endliche finale Teilmenge, d.h. ist $M^\circ$ {\it gut} im Sinne von \cite[p.133]{5}, so ist $\moo$ bereits stark torsion.

\begin{beispiel}
Ist $R$ ein Integrit"atsring, $M$ injektiv und Rang$(M)= $ Rang$(\moo)$, so ist $M$ bereits reflexiv.
\end{beispiel}

\begin{proof}
$T(M)$ ist abgeschlossen, also direkter Summand in $M$, und in $M= T(M)\oplus M_1$ ist dann $M_1 \in \cC,\;\,M_1$ torsionsfrei, also $M_1$ reflexiv. Sei jetzt gleich $M$ injektiv und torsion: Nach Voraussetzung ist $\moo$ torsion, au"serdem $M^\circ$ flach, also gut, so da"s
$\moo$ sogar stark torsion ist und mit (2.1) die Behauptung folgt.
\end{proof}

\begin{beispiel}
Ist $R$ ein diskreter Bewertungsring, so gilt f"ur jeden $R$-Modul $M$:
$$\mbox{Rang}\,(M)\;=\;\mbox{Rang}\,(\moo)\;\Longleftrightarrow\; M=M_1\oplus M_2 \mbox{ mit $M_1$ reflexiv, $M_2$ beschr"ankt.}$$
\end{beispiel}

\begin{proof}
Nur $''\Rightarrow''$ ist zu zeigen, und weil "uber $R$ jeder Modul gut ist, ist $T(M)^{\circ\circ}$ sogar stark torsion, also nach (2.1) $D(T(M))$ artinsch und reflexiv,
$T(M)/D(T(M))$ beschr"ankt. Damit ist $T(M)=D(T(M))\oplus M_2$ rein-injektiv,
$M=T(M)\oplus M_3$, also $M=M_1\oplus M_2$ mit $M_1 =D(T(M))\oplus M_3$ reflexiv und $M_2$ beschr"ankt.
\end{proof}

Sei $R$ wieder beliebig und $M$ reflexiv. Dann folgt aus der Bijektion $\cL(M)\to \cL(M^\circ),\\ U\mapsto $ Ann$_{M^\circ}(U)$, da"s f"ur jeden Untermodul $U$ von $M$ die Menge $\{V \subset M|\,V+U=M\}$
ein minimales Element $V_0$ besitzt, ein sogenanntes (Additions-){\it Komplement} von $U$ in $M$.
Allein aus dieser Komplementeigenschaft folgt sehr viel f"ur die Struktur von $M$:

\setcounter{theorem}{2}
\begin{theorem}
Ist $R$ ein Integrit"atsring mit Quotientenk"orper $K\neq R$ und $M$ ein komplementierter $R$-Modul, so gilt
$$M\;=\;M_1+M_2+D(M),$$
wobei $M_1$ beschr"ankt, $M_2$ endlich erzeugt und $D(M)$ von der Form $D_1 \oplus D_2$ ist mit $D_1$ artinsch, $D_2 \cong K^n\;\,(n\geq 0).$
\end{theorem}

\begin{proof}
Sei $P(M)$ der gr"o"ste radikalvolle Untermodul von $M$.
Weil dann jeder Zwischenmodul $P(M) \subset U \subset M$, mit $M/U$ radikalvoll, ein Komplement in $M$ besitzt, folgt bereits $U=M$,
d.h. $M/P(M)$ ist koatomar. Ein Komplement $V_0$ von $P(M)$ in $M$ ist dann ebenfalls koatomar, also von der Form $V_0=V_1+M_2$ mit $M_2$ endlich erzeugt, $\my^eV_1=0$ f"ur ein $e\geq 1.$

Weil auch $P(M)$ komplementiert ist \cite[Lemma 2.6]{4}, besitzt $P(M)/D(M)$ wie eben keine teilbaren Faktormoduln, ist also kotorsion.
Nach \cite[Theorem 2.4]{3} ist nun $P(M)$ minimax, ein Komplement $W_0$ von $D(M)$ in $P(M)$ also sogar stark kotorsion, au"serdem radikalvoll, d.h. nach \cite[Satz 2.3]{5} schon beschr"ankt.
Damit ist auch $M_1 =W_0+V_1$ beschr"ankt und $M=M_1+M_2+D(M).$

Bleibt $D(M)$ zu zerlegen: Auch $D_1=T(D(M))$ ist minimax, besitzt also einen endlich erzeugten Untermodul $X$, so da"s $D_1/X$
artinsch ist, und aus $X \subset D_1[r]$ mit $0\neq r \in R$ folgt, da"s auch $D_1/D_1[r] \cong rD_1$, also $D_1$ selbst artinsch ist. Weil $D(M)/D_1$ teilbar, torsionsfrei und minimax ist, also isomorph zu $K^n(\geq 0)$, ist $D_1$ als reiner, artinscher Untermodul von $D(M)$ sogar direkter Summand,
$D(M)=D_1\oplus D_2$ und $D_2 \cong K^n$ wie gew"unscht.
\end{proof}

\noindent{{\bf Bemerkung 2.4}\quad
Seien $R$ und $M$ wie in Satz 2.3. Falls $\dim (R)\geq 2$, ist bekanntlich $K$ als $R$-Modul nicht minimax, also $D_2=0$ und daher $D(M)$ selbst artinsch.

%miau

\end{document}